\documentclass[11pt]{article}
\usepackage[a4paper,margin=1in]{geometry}
\usepackage{amsmath,amsthm,amssymb,mathtools}
\usepackage{enumitem}
\usepackage{microtype}
\usepackage[T1]{fontenc}
\usepackage{lmodern}
\usepackage[hidelinks]{hyperref}
\usepackage{tikz}
\usetikzlibrary{decorations.pathreplacing, calc}
\usepackage{subcaption}
\newtheorem{theorem}{Theorem}[section]
\newtheorem{lemma}[theorem]{Lemma}
\newtheorem{proposition}[theorem]{Proposition}

\newtheorem{definition}[theorem]{Definition}

\theoremstyle{remark}

\newcommand{\diam}{\operatorname{diam}}
\newcommand{\DS}{\operatorname{DS}}
\newcommand{\AS}{\mathrm{AS}}

\title{Maximizing the Steklov eigenvalues on trees with a diameter constraint}
\author{
    Jiangdong Ai\thanks{School of Mathematical Sciences and LPMC, Nankai University, Tianjin 300071, China. {\tt jd@nankai.edu.cn}. Partially supported by the Fundamental and Interdisciplinary Disciplines Breakthrough Plan of the Ministry of Education of China (JYB2025XDXM207).}
    \hspace{2mm}
    Huiqiu Lin\thanks{School of Mathematics, East China University of Science and Technology, 130 Meilong Road, Shanghai 200237, China. {\tt{huiqiulin@126.com}.}}
    \hspace{2mm}
    Yongtang Shi\thanks{Center for Combinatorics and LPMC, Nankai University, Tianjin 300071, China. {\tt{shi@nankai.edu.cn.}}
    Partially supported by the National Natural Science Foundation of China (No. 12431013).}
}
\date{}

\begin{document}
\maketitle

\begin{abstract}
We study the first nonzero Steklov eigenvalue $\lambda_2(T,\delta\Omega)$
of the Dirichlet-to-Neumann operator on a finite tree $T$ with leaf boundary $\delta\Omega$, under a constraint on the diameter $D$. He and Hua [Calc. Var. PDE, 2022] showed that $\lambda_2(T) \leq 2/D$ for any tree of diameter $D$, with the even-diameter equality case fully characterized. For odd $D$, the geometric picture underlying the sharp configurations has remained unclear beyond diameter three.
We determine this picture completely for all odd diameters $D = 2r+1 \geq 5$. The sharp value of $\lambda_2$ is achieved on spider trees with nearly-equidistributed branch lengths, forming the family of \emph{generalized almost seesaw trees} $\mathrm{AS}(r,q+2,c,t)$, prescribed by the arithmetic of $n$ relative to $\lceil r/2 \rceil$.
Together with the results of He-Hua and Lin-Zhao [Bull. Lond. Math. Soc., 2025] for even diameters and diameter three, this completes the geometric classification for every diameter.
The argument is based on a scalar root equation for one-center profiles, an inverse boundary quadratic form on boundary fluxes, and a reduction scheme from arbitrary trees to two-center profiles, and then to the
one-center class. The inverse variational viewpoint may be regarded as a boundary analogue of the classical distance-matrix formalism for trees initiated by Graham and Lov\'asz [Adv. Math., 1978].
\end{abstract}

\section{Introduction}
Among classical questions in spectral geometry lies the problem of minimizing or maximizing
eigenvalues under geometric constraints. In the continuous setting, Steklov eigenvalues under a
diameter constraint were studied, for instance, by Al Sayed, Bogosel, Henrot, and Nacry
\cite{AlSayedBogoselHenrotNacry2021} in the convex Euclidean framework. The present paper is
concerned with a discrete analogue of this theme for trees, viewed as one-dimensional singular
domains endowed with their leaf boundary.

Let $T=(V,E)$ be a finite tree and let $\delta\Omega\subset V$ be its leaf set. We regard
$\delta\Omega$ as the boundary of $T$ and denote by $V(T)\setminus\delta\Omega$ the interior.
For a function $f\colon V(T)\to\mathbb R$, the discrete Laplace operator is defined by
$$
\Delta f(v):=\sum_{w\sim v}\bigl(f(v)-f(w)\bigr),\qquad v\in V(T),
$$
and, for a boundary leaf $\omega\in\delta\Omega$ with unique neighbor $\omega'$, the outward
normal derivative is
$$
\partial_n f(\omega):=f(\omega)-f(\omega').
$$
A real number $\lambda\ge0$ is called a Steklov eigenvalue of $(T,\delta\Omega)$ if there exists a
nontrivial function $f\colon V(T)\to\mathbb R$ such that
$$
\Delta f(v)=0 \quad \text{for } v\in V(T)\setminus\delta\Omega,
\qquad
\partial_n f(\omega)=\lambda f(\omega) \quad \text{for } \omega\in\delta\Omega.
$$
Thus harmonicity is imposed at interior points, while the Steklov condition is prescribed on the
boundary leaves. The corresponding spectrum is a finite nondecreasing sequence
$$
0=\lambda_1(T,\delta\Omega)\le \lambda_2(T,\delta\Omega)\le \cdots \le
\lambda_{|\delta\Omega|}(T,\delta\Omega),
$$
 we abbreviate $\lambda_k(T,\delta\Omega)$ as $\lambda_k(T)$ or simply $\lambda_k$ when there is no
confusion. Given $g\in\mathbb R^{\delta\Omega}$, its harmonic extension $\widehat g$ is the
unique function on $V(T)$ with $\widehat g|_{\delta\Omega}=g$ and $\Delta\widehat g=0$ on the
interior. The associated Dirichlet-to-Neumann operator is
$$
\Lambda_T g := (\partial_n \widehat g)|_{\delta\Omega}.
$$
The first nonzero Steklov eigenvalue then admits the variational characterization
$$
\lambda_2(T,\delta\Omega)
=
\min_{\substack{g\in\mathbb R^{\delta\Omega}\setminus\{0\}\\
\sum_{\omega\in\delta\Omega} g(\omega)=0}}
\frac{\sum_{xy\in E(T)}\bigl(\widehat g(x)-\widehat g(y)\bigr)^2}{\sum_{\omega\in\delta\Omega} g(\omega)^2}.
$$
In particular, the problem is a diameter-constrained maximization problem for a boundary spectral
quantity on a discrete one-dimensional domain.

In this paper we study
$$
\max\{\lambda_2(T,\delta\Omega): |V(T)|=n,\ \diam(T)=D\}
$$
when the diameter $D$ is odd. The Steklov problem on manifolds goes back to Stekloff and is now one of
the central topics in spectral geometry; see the surveys of Girouard--Polterovich and
Colbois--Girouard--Gordon--Sher
\cite{Stekloff1902,GirouardPolterovich2017,ColboisGirouardGordonSher2024}.
Its discrete counterpart on graphs, via Dirichlet-to-Neumann operators, has developed rapidly in
recent years; representative references include
\cite{HuaHuangWang2017,HassannezhadMiclo2020,Perrin2019,Perrin2021,Tschanz2022,HanHua2023,ShiYuComp2022,ShiYuLich2022,YuYu2024,LinZhaoPlanar2025,LLYZ-2026}.
For trees, He and Hua established sharp diameter upper bounds for Steklov eigenvalues
\cite{HeHua2022}, Lin and Zhao determined the exact extremal picture for even diameter and for
diameter $3$ \cite{LinZhaoTrees2025}, and more recently Ai, Ji, Lian, and Yang~\cite{AiJiLianYang2026} proved a sharp
upper bound in terms of the maximum degree and the number of leaves for leaf-boundary trees. 
The remaining open case in the prescribed-diameter problem is therefore the odd-diameter regime $D=2r+1\ge5$. The purpose of the present paper is to resolve this case.

A complementary viewpoint is provided by the distance matrix of a tree. Classical works of
Graham--Pollak, Graham--Lov\'asz, and Merris show that several spectral questions on trees can be
reformulated in terms of quadratic forms built from the distance matrix~\cite{GrahamPollak1971,GrahamLovasz1978,Merris1990}. In our setting, the inverse Steklov
variational principle naturally leads to a quadratic form on boundary fluxes which admits the
representation
$$
Q_T(z)=-\frac12 z^T D_{\delta\Omega}z,
$$
where $D_{\delta\Omega}$ is the distance matrix of the leaf set. In this sense, the present problem
may be viewed as a boundary variant of the classical distance-matrix formalism for trees.

Write $D=2r+1$ and $M=n-2r-2$. The geometry of a spider is completely encoded by its branch-length
profile, so the one-center problem becomes a finite-dimensional length redistribution problem.
A generalized almost seesaw tree $AS(r,b,c,t)$ is the spider whose branch lengths are
$r+1$, $r$, exactly $t$ copies of $c+1$, and $b-2-t$ copies of $c$.
Its precise role in the extremal problem is described in Theorem~\ref{thm:global-odd} below.

Our main conclusions can be stated as follows.
\begin{theorem}
	\label{thm:all-diameters}
	Let $D\ge 2$ and let $n\ge D+1$. Consider the problem of maximizing $\lambda_2$
	among trees with order $n$ and diameter $D$.
	Then the extremizers are completely determined for every diameter.
	
	More precisely, the even-diameter case and the case $D=3$ were determined by
	Lin and Zhao~\cite{LinZhaoTrees2025}, while the remaining odd-diameter case $D\ge 5$
	is resolved in the present paper.
\end{theorem}
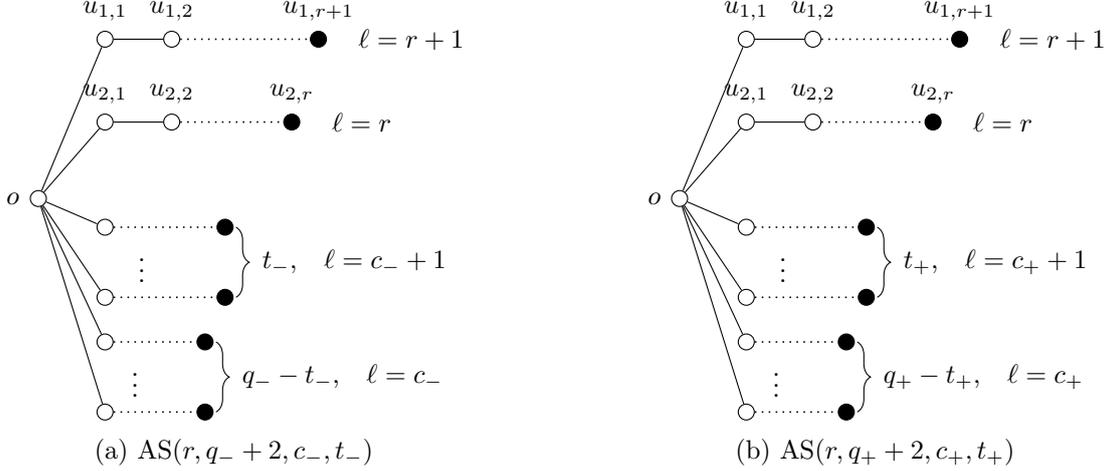
\begin{figure}[htbp]
\centering

\begin{subfigure}[b]{0.47\textwidth}
\centering
\begin{tikzpicture}[scale=0.65,
  v/.style  = {circle, draw, fill=white, inner sep=0pt, minimum size=6pt},
  lf/.style = {circle, draw, fill=black,  inner sep=0pt, minimum size=6pt},
  lbl/.style = {font=\small},
  dt/.style  = {dotted, semithick}
]
\def\H{1.35}\def\V{1.3}

\node[v, label={[lbl]left:$o$}] (o) at (0,0) {};

\node[v,  label={[lbl]above:$u_{1,1}$}]   (p11) at (  \H, 2.5*\V) {};
\node[v,  label={[lbl]above:$u_{1,2}$}]   (p12) at (2*\H, 2.5*\V) {};
\node[lf, label={[lbl]above:$u_{1,r+1}$}] (p1e) at (4.2*\H, 2.5*\V) {};
\draw(o)--(p11)--(p12);
\draw[dt](p12)--(p1e);
\node[lbl, anchor=west, xshift=8pt] at (p1e.east) {$\ell=r+1$};

\node[v,  label={[lbl]above:$u_{2,1}$}]  (p21) at (  \H, 1.2*\V) {};
\node[v,  label={[lbl]above:$u_{2,2}$}]  (p22) at (2*\H, 1.2*\V) {};
\node[lf, label={[lbl]above:$u_{2,r}$}]  (p2e) at (3.8*\H, 1.2*\V) {};
\draw(o)--(p21)--(p22);
\draw[dt](p22)--(p2e);
\node[lbl, anchor=west, xshift=8pt] at (p2e.east) {$\ell=r$};

\node[v]  (a1t) at (  \H, -0.45*\V) {};
\node[lf] (a1e) at (2.8*\H, -0.45*\V) {};
\draw(o)--(a1t);
\draw[dt](a1t)--(a1e);

\node[lbl] at (1.55*\H, -1.00*\V) {$\vdots$};

\node[v]  (a1b) at (  \H, -1.55*\V) {};
\node[lf] (a1f) at (2.8*\H, -1.55*\V) {};
\draw(o)--(a1b);
\draw[dt](a1b)--(a1f);

\draw[decorate, decoration={brace, amplitude=5pt}]
  ($(a1e)+(0.22,0)$) -- ($(a1f)+(0.22,0)$)
  node[midway, right=6pt, lbl] {$t_{-}$,\quad$\ell=c_{-}+1$};

\node[v]  (a2t) at (  \H, -2.25*\V) {};
\node[lf] (a2e) at (2.5*\H, -2.25*\V) {};
\draw(o)--(a2t);
\draw[dt](a2t)--(a2e);

\node[lbl] at (1.45*\H, -2.80*\V) {$\vdots$};

\node[v]  (a2b) at (  \H, -3.35*\V) {};
\node[lf] (a2f) at (2.5*\H, -3.35*\V) {};
\draw(o)--(a2b);
\draw[dt](a2b)--(a2f);

\draw[decorate, decoration={brace, amplitude=5pt}]
  ($(a2e)+(0.22,0)$) -- ($(a2f)+(0.22,0)$)
  node[midway, right=6pt, lbl] {$q_{-}-t_{-}$,\quad$\ell=c_{-}$};

\end{tikzpicture}
\caption{$\mathrm{AS}(r,q_{-}+2,c_{-},t_{-})$}
\end{subfigure}
\hfill
\begin{subfigure}[b]{0.47\textwidth}
\centering
\begin{tikzpicture}[scale=0.65,
  v/.style  = {circle, draw, fill=white, inner sep=0pt, minimum size=6pt},
  lf/.style = {circle, draw, fill=black,  inner sep=0pt, minimum size=6pt},
  lbl/.style = {font=\small},
  dt/.style  = {dotted, semithick}
]
\def\H{1.35}\def\V{1.3}

\node[v, label={[lbl]left:$o$}] (o) at (0,0) {};

\node[v,  label={[lbl]above:$u_{1,1}$}]   (q11) at (  \H, 2.5*\V) {};
\node[v,  label={[lbl]above:$u_{1,2}$}]   (q12) at (2*\H, 2.5*\V) {};
\node[lf, label={[lbl]above:$u_{1,r+1}$}] (q1e) at (4.2*\H, 2.5*\V) {};
\draw(o)--(q11)--(q12);
\draw[dt](q12)--(q1e);
\node[lbl, anchor=west, xshift=8pt] at (q1e.east) {$\ell=r+1$};

\node[v,  label={[lbl]above:$u_{2,1}$}]  (q21) at (  \H, 1.2*\V) {};
\node[v,  label={[lbl]above:$u_{2,2}$}]  (q22) at (2*\H, 1.2*\V) {};
\node[lf, label={[lbl]above:$u_{2,r}$}]  (q2e) at (3.8*\H, 1.2*\V) {};
\draw(o)--(q21)--(q22);
\draw[dt](q22)--(q2e);
\node[lbl, anchor=west, xshift=8pt] at (q2e.east) {$\ell=r$};

\node[v]  (b1t) at (  \H, -0.45*\V) {};
\node[lf] (b1e) at (2.8*\H, -0.45*\V) {};
\draw(o)--(b1t);
\draw[dt](b1t)--(b1e);

\node[lbl] at (1.55*\H, -1.00*\V) {$\vdots$};

\node[v]  (b1b) at (  \H, -1.55*\V) {};
\node[lf] (b1f) at (2.8*\H, -1.55*\V) {};
\draw(o)--(b1b);
\draw[dt](b1b)--(b1f);

\draw[decorate, decoration={brace, amplitude=5pt}]
  ($(b1e)+(0.22,0)$) -- ($(b1f)+(0.22,0)$)
  node[midway, right=6pt, lbl] {$t_{+}$,\quad$\ell=c_{+}+1$};

\node[v]  (b2t) at (  \H, -2.25*\V) {};
\node[lf] (b2e) at (2.5*\H, -2.25*\V) {};
\draw(o)--(b2t);
\draw[dt](b2t)--(b2e);

\node[lbl] at (1.45*\H, -2.80*\V) {$\vdots$};

\node[v]  (b2b) at (  \H, -3.35*\V) {};
\node[lf] (b2f) at (2.5*\H, -3.35*\V) {};
\draw(o)--(b2b);
\draw[dt](b2b)--(b2f);

\draw[decorate, decoration={brace, amplitude=5pt}]
  ($(b2e)+(0.22,0)$) -- ($(b2f)+(0.22,0)$)
  node[midway, right=6pt, lbl] {$q_{+}-t_{+}$,\quad$\ell=c_{+}$};

\end{tikzpicture}
\caption{$\mathrm{AS}(r,q_{+}+2,c_{+},t_{+})$}
\end{subfigure}

\caption{The two explicit extremal candidates in Theorem~1.2 and Theorem~4.1, where
$M=q_{-}c_{-}+t_{-}=q_{+}c_{+}+t_{+}$.}
\end{figure}
\begin{theorem}\label{thm:odd-main}
	Let $D=2r+1\ge 3$ be odd and let $n\ge D+1$. Let $
M=n-2r-2$, $s=\Bigl\lceil \frac r2\Bigr\rceil$, $q_- = \max\Bigl\{1,\Bigl\lfloor \frac Ms\Bigr\rfloor\Bigr\}$ and $q_+ = \Bigl\lceil \frac Ms\Bigr\rceil.$
If $M=0$, then the unique extremizer is the path $P_{D+1}$. Assume $M\ge 1$, let $M=q_-c_-+t_-$ with $ 0\le t_-<q_-$, and $M=q_+c_++t_+$ with $0\le t_+<q_+$.
Then every extremizer of $\lambda_2$ among trees with order $n$ and diameter $D$ is
isomorphic to one of the two generalized almost seesaw trees
$$
\AS(r,q_-+2,c_-,t_-),
\qquad
\AS(r,q_++2,c_+,t_+).
$$
In particular, if these two trees have distinct first nonzero Steklov eigenvalues, then the extremizer is unique.
For each fixed odd diameter, the extremizer is unique for all sufficiently large orders.
\end{theorem}
A more explicit order-by-order formulation is given later in Theorem~\ref{thm:all-orders}. The
structural core of the proof can be isolated in the following reduction principle.

\begin{theorem}\label{thm:reduction-principle}
Let $D=2r+1\ge 3$ be odd and let $n\ge D+1$. For every tree $T$ with $|V(T)|=n$ and
$\diam(T)=D$, there exists a double spider of the same order and diameter such that $\lambda_2(T)\le \lambda_2(D_T).$ If equality holds, then $T$ itself is a double spider. 

Moreover, every extremal double spider is a spider.
In particular, every extremizer of $\lambda_2$ among trees with order $n$ and diameter $D$
is a spider.
\end{theorem}

The proof has three layers. In Section~\ref{sec:spider-family} we work in the one-center class and
encode a spider by its branch-length profile. The first nonzero Steklov eigenvalue is then
characterized as the unique root of a scalar rational function, which allows us to balance the lateral and principal branches and to reduce the one-center optimization to a unimodal scalar parameter.

The passage from arbitrary trees to one-center profiles is variational. In Section~\ref{sec:reduction} we introduce an inverse boundary quadratic form on mean-zero
boundary fluxes. A cut decomposition shows that this form can be written both as a sum of square cut fluxes and as $-\frac12 z^T D_{\delta\Omega} z$, thereby connecting the Steklov problem to the classical distance-matrix theory of trees. We then exploit the sign pattern of maximizing fluxes on two-center profiles to compare arbitrary trees with suitable double spiders and to transport nonprincipal branches from one side to the other. This yields the one-center rigidity embodied in Theorem~\ref{thm:reduction-principle}.

The paper is organized as follows. Section~\ref{sec:spider-family} is devoted to one-center profiles:
we introduce the generalized almost seesaw trees, derive the one-variable root equation, and prove
both the exact optimization theorem and the unimodality in the profile parameter. In
Section~\ref{sec:reduction} we reformulate the problem in terms of boundary fluxes and inverse quadratic
forms, and we compare arbitrary trees with suitable two-center profiles. In Section~\ref{sec:global-classification}
we combine these ingredients to obtain the global two-profile reduction, the single-threshold
comparison, and the final order-by-order formulation.

\section{Profile optimization in the one-center class}\label{sec:spider-family}

\subsection{One-center profiles and generalized almost seesaw trees}
A spider $S(\ell_1,\dots,\ell_b)$ is the tree obtained by gluing together one endpoint of each of $b$
pairwise internally disjoint paths of lengths $\ell_1,\dots,\ell_b$. The common endpoint is called
the center, and the numbers $\ell_1,\dots,\ell_b$ are the branch lengths.

\begin{definition}
Let $r\ge 1$, $b\ge 2$, $c\ge 1$ be integers, and let $0\le t\le b-2$. The generalized almost
seesaw tree $\AS(r,b,c,t)$ is the spider obtained by gluing together one endpoint of $b$ paths,
where one path has length $r$, one path has length $r+1$, exactly $t$ of the remaining $b-2$
paths have length $c+1$, and the other $b-2-t$ paths have length $c$.
\end{definition}

If $M\ge 1$, $\lceil M/r\rceil\le q\le M$, and
$$
M=qc+t,\qquad 0\le t<q,
$$
then the corresponding generalized almost seesaw tree is
$$
\AS(r,q+2,c,t)=S\bigl(r+1,r,\underbrace{c+1,\dots,c+1}_{t},\underbrace{c,\dots,c}_{q-t}\bigr).
$$

The geometry of a spider is completely encoded by its branch-length profile. This turns the diameter-constrained problem in the one-center class into a finite-dimensional redistribution problem. We begin with the exact eigenvalue equation for such a profile.

\begin{lemma}\label{lem:spider-equation}
Let $T=S(\ell_1,\dots,\ell_b)$
be a spider with $\ell_1>\ell_2\ge \ell_3\ge \cdots \ge \ell_b\ge 1.$
Define $$
F_{\ell}(\lambda):=\sum_{i=1}^b \frac{1}{1-\ell_i\lambda}.
$$
Then the first nonzero Steklov eigenvalue $\lambda_2(T)$ is the unique root of $F_{\ell}$
in the interval $\left(\frac1{\ell_1},\frac1{\ell_2}\right).$
\end{lemma}

\begin{proof}
Let $o$ be the center of the spider.
For each branch, write $
o=v_{i,0}\sim v_{i,1}\sim \cdots \sim v_{i,\ell_i},$
where $v_{i,\ell_i}$ is the leaf at the end of the $i$-th branch.
Let $x_i$ be the boundary value at $v_{i,\ell_i}$, and let $m=f(o)$
be the value at the center.
Since the harmonic extension is affine on each path,
$f(v_{i,k})=m+\frac{k}{\ell_i}(x_i-m), 0\le k\le \ell_i.$
Hence the normal derivative at the leaf $v_{i,\ell_i}$ is
$$
\partial_n f(v_{i,\ell_i})
=
 f(v_{i,\ell_i})-f(v_{i,\ell_i-1})
=
\frac{x_i-m}{\ell_i}.
$$
Therefore the Steklov eigenvalue equation reads
\begin{equation}\label{eq:boundary-eq}
\frac{x_i-m}{\ell_i}=\lambda x_i,
\qquad i=1,\dots,b.
\end{equation}

Now the center is harmonic, so
$$
\sum_{i=1}^b \bigl(f(v_{i,1})-m\bigr)=0.
$$
Using the affine formula on each branch, this is equivalent to
$$
\sum_{i=1}^b \frac{x_i-m}{\ell_i}=0.
$$
By \eqref{eq:boundary-eq} we obtain
$$
\lambda\sum_{i=1}^b x_i=0.
$$
Thus, for any nonzero eigenvalue, $\sum_i x_i=0$.

Assume now that $m\neq 0$.
Then \eqref{eq:boundary-eq} gives
$x_i=\frac{m}{1-\ell_i\lambda}.$
Substituting this into $\sum_i x_i=0$ gives
$$
\sum_{i=1}^b \frac{1}{1-\ell_i\lambda}=0,
$$
namely
$F_{\ell}(\lambda)=0.$

Next,
$$
F_{\ell}'(\lambda)=\sum_{i=1}^b \frac{\ell_i}{(1-\ell_i\lambda)^2}>0,
$$
so $F_{\ell}$ is strictly increasing on every interval avoiding its poles.
Because $\ell_1>\ell_2$, we have
$$
\lim_{\lambda\to (1/\ell_1)^+} F_{\ell}(\lambda)=-\infty,
\qquad
\lim_{\lambda\to (1/\ell_2)^-} F_{\ell}(\lambda)=+\infty.
$$
Hence $F_{\ell}$ has a unique root in
$\bigl(\frac1{\ell_1},\frac1{\ell_2}\bigr)$.

It remains to show that this root is the \emph{first} positive Steklov eigenvalue.
If $m=0$, then \eqref{eq:boundary-eq} becomes
$\frac{x_i}{\ell_i}=\lambda x_i.$
Thus every nonzero component $x_i$ forces $\lambda=1/\ell_i$.
If $\lambda=1/\ell_1$, then because $\ell_1>\ell_2$ we must have $x_i=0$ for every $i\ge 2$.
But then $\sum_i x_i=0$ forces $x_1=0$, contradicting that the eigenvector is nonzero.
Therefore any eigenvalue arising from $m=0$ is of the form $1/\ell_i$ for some $i\ge 2$, and hence
is at least $1/\ell_2$.
So there is no positive Steklov eigenvalue below $1/\ell_2$ except the unique root of
$F_{\ell}$ found above.
This proves that this root is exactly $\lambda_2(T)$.
\end{proof}

\subsection{Exact optimization in the one-center class}
We now show that, for fixed principal branches, the remaining branches must be as equal as possible.

\begin{lemma}\label{lem:balance-side}
Fix $r\ge 1$ and let
$$
T_0=S(r+1,\,r,\,k,\,s,\,m_5,\dots,m_b)
$$
be a spider such that
$$
r\ge u\ge v+2\ge 3,
\qquad
r\ge m_j\ge 1 \quad (5\le j\le b).
$$
Let
$$
T_1=S(r+1,\,r,\,k-1,\,s+1,\,m_5,\dots,m_b).
$$
Then $\lambda_2(T_1)>\lambda_2(T_0)$.
\end{lemma}

\begin{proof}
For $t\in[0,1]$, define
$$
F_t(\lambda):=
\frac1{1-(r+1)\lambda}
+
\frac1{1-r\lambda}
+
\frac1{1-(k-t)\lambda}
+
\frac1{1-(s+t)\lambda}
+
\sum_{j=5}^b \frac1{1-m_j\lambda}.
$$
Since $k-t\le r$ and $s+t\le k-t\le r$, all terms other than the first two are regular on
$\bigl(\frac1{r+1},\frac1r\bigr)$, and
$$
\partial_\lambda F_t(\lambda)
=
\frac{r+1}{(1-(r+1)\lambda)^2}
+
\frac{r}{(1-r\lambda)^2}
+
\frac{k-t}{(1-(k-t)\lambda)^2}
+
\frac{s+t}{(1-(s+t)\lambda)^2}
+
\sum_{j=5}^b \frac{m_j}{(1-m_j\lambda)^2}
>0.
$$
Moreover,
$$
\lim_{\lambda\to (1/(r+1))^+}F_t(\lambda)=-\infty,
\qquad
\lim_{\lambda\to (1/r)^-}F_t(\lambda)=+\infty.
$$
Hence $F_t$ has a unique root
$$
\lambda(t)\in\left(\frac1{r+1},\frac1r\right).
$$
For $t=0$ and $t=1$, Lemma~\ref{lem:spider-equation} gives
$\lambda(0)=\lambda_2(T_0), \qquad \lambda(1)=\lambda_2(T_1).$

By the implicit function theorem,
$\lambda'(t)=-\frac{\partial_t F_t(\lambda(t))}{\partial_\lambda F_t(\lambda(t))}.$
Now
$$
\partial_t F_t(\lambda)
=
-\frac{\lambda}{\bigl(1-(k-t)\lambda\bigr)^2}
+
\frac{\lambda}{\bigl(1-(s+t)\lambda\bigr)^2},
$$
so
$$
\lambda'(t)
=
\frac{\lambda(t)}{\partial_\lambda F_t(\lambda(t))}
\left(
\frac{1}{\bigl(1-(k-t)\lambda(t)\bigr)^2}
-
\frac{1}{\bigl(1-(s+t)\lambda(t)\bigr)^2}
\right).
$$
Because $k-t\ge s+t$ for $t\in[0,1]$, with equality only possibly at $t=1$ when $k=s+2$, and because
$\lambda(t)<1/r\le 1/(k-t)$, we have
$$
0<1-(k-t)\lambda(t)\le 1-(s+t)\lambda(t).
$$
For $t\in[0,1)$ the inequality is strict, hence
$$
\frac{1}{\bigl(1-(k-t)\lambda(t)\bigr)^2}
>
\frac{1}{\bigl(1-(s+t)\lambda(t)\bigr)^2}.
$$
Therefore $\lambda'(t)>0$ on $[0,1)$, and consequently
$\lambda_2(T_1)=\lambda(1)>\lambda(0)=\lambda_2(T_0).$
\end{proof}

We next show that, once at least one additional branch is present, the two longest branches must be as
balanced as the odd diameter permits.

\begin{lemma}\label{lem:balance-main}
Let
$T_0=S(a,\,b,\,m_3,\dots,m_s)$
be a spider such that
$a+b=2r+1, a\ge b+2, 1\le m_j\le b$ where $3\le j\le s$,
and assume $s\ge 3$.
Let
$T_1=S(a-1,\,b+1,\,m_3,\dots,m_s).$
Then
$\lambda_2(T_1)>\lambda_2(T_0).$
\end{lemma}

\begin{proof}
For $t\in[0,1]$, define
$$
F_t(\lambda)=
\frac1{1-(a-t)\lambda}
+
\frac1{1-(b+t)\lambda}
+
\sum_{j=3}^s \frac1{1-m_j\lambda}.
$$
Since $a+b$ is odd and $a\ge b+2$, we in fact have $a\ge b+3$.
Hence $a-t>b+t\ge m_j$ for every $j$ and every $t\in[0,1]$, so the function $F_t$ is strictly increasing on
$$
\left(\frac1{a-t},\frac1{b+t}\right),
$$
with
$$
\lim_{\lambda\to (1/(a-t))^+}F_t(\lambda)=-\infty,
\qquad
\lim_{\lambda\to (1/(b+t))^-}F_t(\lambda)=+\infty.
$$
Hence $F_t$ has a unique root
$$
\lambda(t)\in\left(\frac1{a-t},\frac1{b+t}\right).
$$
For $t=0$ and $t=1$, Lemma~\ref{lem:spider-equation} yields
$\lambda(0)=\lambda_2(T_0)$ and $\lambda(1)=\lambda_2(T_1).$

We first prove that
\begin{equation}\label{eq:sigma-less-than-2D}
\lambda(t)<\frac{2}{2r+1}
\qquad\text{for all }t\in[0,1].
\end{equation}
Indeed, since $(a-t)+(b+t)=2r+1$, the first two terms cancel at $\lambda=\frac{2}{2r+1}$:
$\frac1{1-\frac{2(a-t)}{2r+1}} + \frac1{1-\frac{2(b+t)}{2r+1}} =0.$
Moreover, each additional branch satisfies $m_j\le b<\frac{2r+1}{2}$, so
$1-\frac{2m_j}{2r+1}>0.$
Therefore
$$
F_t\!\left(\frac{2}{2r+1}\right)>0.
$$
Since $F_t$ is strictly increasing and $F_t(\lambda(t))=0$, this proves \eqref{eq:sigma-less-than-2D}.

By the implicit function theorem,
$\lambda'(t)=-\frac{\partial_t F_t(\lambda(t))}{\partial_\lambda F_t(\lambda(t))}.$
A direct differentiation gives
$$
\lambda'(t)
=
\frac{\lambda(t)}{\partial_\lambda F_t(\lambda(t))}
\left(
\frac{1}{\bigl(1-(a-t)\lambda(t)\bigr)^2}
-
\frac{1}{\bigl(1-(b+t)\lambda(t)\bigr)^2}
\right).
$$
Because
$\frac1{a-t}<\lambda(t)<\frac1{b+t},$
we have
$1-(a-t)\lambda(t)<0<1-(b+t)\lambda(t).$
Thus it suffices to compare absolute values. Using \eqref{eq:sigma-less-than-2D},
$$
(a-t)\lambda(t)-1
<
1-(b+t)\lambda(t)
\iff
\bigl((a-t)+(b+t)\bigr)\lambda(t)<2
\iff
(2r+1)\lambda(t)<2,
$$
which is true. Therefore
$\bigl|1-(a-t)\lambda(t)\bigr|<1-(b+t)\lambda(t),$
and hence
$$
\frac{1}{\bigl(1-(a-t)\lambda(t)\bigr)^2}
>
\frac{1}{\bigl(1-(b+t)\lambda(t)\bigr)^2}.
$$
Consequently $\lambda'(t)>0$ on $[0,1]$, so
$\lambda_2(T_1)=\lambda(1)>\lambda(0)=\lambda_2(T_0).$
\end{proof}

We now combine the previous ingredients to get the following result.

\begin{theorem}\label{thm:main}
Let $D=2r+1\ge 3$ be odd and let $n\ge D+1$.
Let $M=n-2r-2.$
Among all spiders with diameter $D$ and order $n$, the following hold.
\begin{enumerate}[label=\textnormal{(\roman*)}]
    \item If $M=0$, then the unique spider is the path $P_{D+1}$.
    \item Assume $M\ge 1$. For each admissible integer $q$ with $\lceil M/r\rceil\le q\le M$, write
    $$
    M=qc+t,\qquad 0\le t<q.
    $$
    Then
    $$
    \max\{\lambda_2(T): T\text{ is a spider},\ \diam(T)=D,\ |V(T)|=n\}
    =
    \max_{\lceil M/r\rceil\le q\le M} \lambda_2\bigl(\AS(r,q+2,c,t)\bigr).
    $$
    For fixed $q$, the unique maximizing spider is $\AS(r,q+2,c,t)$, where $M=qc+t$ and $0\le t<q$.
\end{enumerate}
Its first nonzero Steklov eigenvalue is the unique root
$\lambda\in\bigl(\frac1{r+1},\frac1r\bigr)$ of
\begin{equation}\label{eq:balanced-root}
\frac{1}{1-(r+1)\lambda}+
\frac{1}{1-r\lambda}+
\frac{t}{1-(c+1)\lambda}+
\frac{q-t}{1-c\lambda}=0.
\end{equation}
\end{theorem}

\begin{proof}
If $M=0$, then $n=D+1$, and the only spider of diameter $D$ is the path $P_{D+1}$.
This proves part~\textnormal{(i)}.

Assume henceforth that $M\ge1$ and let
$$
T=S(\ell_1,\dots,\ell_b)
$$
be a spider with diameter $D=2r+1$ and order $n$. Reorder the branch lengths so that
$$
\ell_1>\ell_2\ge \ell_3\ge \cdots \ge \ell_b\ge 1.
$$
Since the diameter is $2r+1$, the two longest branches satisfy
$$
\ell_1+\ell_2=2r+1.
$$
Write
$$
m_1:=\ell_3,\dots,m_q:=\ell_b,
\qquad q:=b-2.
$$
Then
$$
q\ge1,
\qquad
1\le m_j\le \ell_2\le r,
\qquad
m_1+\cdots+m_q=n-2r-2=M.
$$
Thus the diameter constraint fixes the sum of the two principal branches, while the order
constraint fixes the total lateral length.

\smallskip
\noindent\emph{Step 1: balancing the principal profile.}
If $\ell_1\neq r+1$, then $\ell_1\ge \ell_2+2$, and Lemma~\ref{lem:balance-main} shows that
replacing $(\ell_1,\ell_2)$ by $(\ell_1-1,\ell_2+1)$ strictly increases $\lambda_2$. Repeating this
operation finitely many times, we arrive at a spider of the form
$$
S(r+1,r,m_1,\dots,m_q)
$$
with strictly larger first nonzero Steklov eigenvalue. Hence every maximizing spider must have
principal branch lengths $r+1$ and $r$.

\smallskip
\noindent\emph{Step 2: balancing the lateral profile.}
Fix $q$. Among all spiders of the form
$$
S(r+1,r,m_1,\dots,m_q)
$$
with $1\le m_j\le r$ and $m_1+\cdots+m_q=M$, suppose two lateral branches differ by at least
$2$, say $u\ge v+2$. Lemma~\ref{lem:balance-side} then replaces $(u,v)$ by $(k-1,v+1)$ and
strictly increases $\lambda_2$. Iterating this elementary redistribution, we reach a unique profile in
which every two lateral branches differ by at most $1$. Writing
$$
M=qc+t,
\qquad 0\le t<q,
$$
this balanced lateral profile is precisely
$$
(\underbrace{c+1,\dots,c+1}_{t},\underbrace{c,\dots,c}_{q-t}).
$$
Therefore the unique maximizer for fixed $q$ is
$$
\AS(r,q+2,c,t)=S\bigl(r+1,r,\underbrace{c+1,\dots,c+1}_{t},\underbrace{c,\dots,c}_{q-t}\bigr).
$$

\smallskip
\noindent\emph{Step 3: admissible range of the profile parameter.}
The integer $q$ is feasible exactly when each lateral branch has length between $1$ and $r$, namely
$$
\left\lceil \frac{M}{r}\right\rceil\le q\le M.
$$
Taking the maximum over all feasible $q$ proves part~\textnormal{(ii)}. The root equation
\eqref{eq:balanced-root} follows directly from Lemma~\ref{lem:spider-equation}.
\end{proof}

\subsection{Unimodality of the profile parameter}
Theorem~\ref{thm:main} reduces the one-center problem to a finite optimization over the number $q$ of lateral branches. The next result shows that the resulting scalar parameter enters through a genuinely unimodal profile.

Fix $r\ge 1$ and $M\ge 1$.
For every real number $
q\in \left[\frac{M}{r},\,M\right]$,
let
$c(q):=\left\lfloor \frac{M}{q}\right\rfloor.$
Define
\begin{equation}\label{eq:Phi-rMq}
\Phi_{r,M}(\lambda,q)=
\frac{1}{1-(r+1)\lambda}
+
\frac{1}{1-r\lambda}
+
\frac{M-c(q)q}{1-(c(q)+1)\lambda}
+
\frac{q(c(q)+1)-M}{1-c(q)\lambda}.
\end{equation}
Whenever $q$ is a feasible integer, this is exactly the balanced-spider equation
\eqref{eq:balanced-root} for $\AS(r,q+2,c,t)$.

\begin{theorem}\label{thm:q-unimodal}
Let $r\ge 1$, $M\ge 1$, and let
$s=\left\lceil \frac{r}{2}\right\rceil.$
For every $q\in \left[\frac{M}{r},\,M\right]$, the equation
$\Phi_{r,M}(\lambda,q)=0$
has a unique solution $\Sigma_{r,M}(q)\in \left(\frac{1}{r+1},\,\frac{1}{r}\right).$
The function
$q\longmapsto \Sigma_{r,M}(q)$
is continuous on $\bigl[\frac{M}{r},M\bigr]$, piecewise $C^1$, strictly increasing on
$\left[\frac{M}{r},\,\frac{M}{s}\right]$,
and strictly decreasing on
$\left[\frac{M}{s},\,M\right]$.
Consequently, if
$q_-:=\max\!\left\{1,\left\lfloor \frac{M}{s}\right\rfloor\right\}$ and
$q_+:=\left\lceil \frac{M}{s}\right\rceil$,
and if
$$
M=q_-c_-+t_-,\qquad 0\le t_-<q_-,
$$
$$
M=q_+c_++t_+,\qquad 0\le t_+<q_+,
$$
then
\begin{equation}\label{eq:q-two-candidates}
\max_{\lceil M/r\rceil\le q\le M} \lambda_2\bigl(\AS(r,q+2,c,t)\bigr)
=
\max\Bigl\{\lambda_2\bigl(\AS(r,q_-+2,c_-,t_-)\bigr),\,\lambda_2\bigl(\AS(r,q_++2,c_+,t_+)\bigr)\Bigr\},
\end{equation}
where on the left-hand side, for each admissible integer $q$, we write $M=qc+t$ with $0\le t<q$.
\end{theorem}

\begin{proof}
Fix $q\in [M/r,M]$ and write $c:=c(q)=\lfloor M/q\rfloor$.
If $q=M/r$, then necessarily $c=r$, and \eqref{eq:Phi-rMq} becomes
$\Phi_{r,M}(\lambda,M/r) = \frac{1}{1-(r+1)\lambda} + \frac{1+M/r}{1-r\lambda}.$
This function is strictly increasing on $(1/(r+1),1/r)$ and has limits $-\infty$ and $+\infty$
at the two endpoints, hence it has a unique root in that interval.

Assume now that $q>M/r$.
Then $c\le r-1$. The endpoint $q=M$ is covered by the same formula with $c(q)=1$.
On the open interval
$I_c:=\left(\frac{M}{c+1},\,\frac{M}{c}\right)$
the integer $c(q)$ is constant and equal to $c$, so \eqref{eq:Phi-rMq} becomes
$$
\Phi_{r,M}(\lambda,q)
=
\frac{1}{1-(r+1)\lambda}
+
\frac{1}{1-r\lambda}
+
\frac{M-cq}{1-(c+1)\lambda}
+
\frac{q(c+1)-M}{1-c\lambda}.
$$
Since $c\le r-1$, all side denominators are positive on $(1/(r+1),1/r)$.
Moreover,
$$
\partial_\lambda \Phi_{r,M}(\lambda,q)
=
\frac{r+1}{(1-(r+1)\lambda)^2}
+
\frac{r}{(1-r\lambda)^2}
+
\frac{(M-cq)(c+1)}{(1-(c+1)\lambda)^2}
+
\frac{(q(c+1)-M)c}{(1-c\lambda)^2}
>0.
$$
Thus $\Phi_{r,M}(\cdot,q)$ is strictly increasing on $(1/(r+1),1/r)$.
As $\lambda\downarrow 1/(r+1)$, the first term tends to $-\infty$; as
$\lambda\uparrow 1/r$, the second term tends to $+\infty$.
Therefore there is a unique root
$$
\Sigma_{r,M}(q)\in \left(\frac{1}{r+1},\,\frac{1}{r}\right)
$$
for every $q\in [M/r,M]$.

On each open interval $I_c$, the implicit function theorem gives a $C^1$ branch
$q\mapsto \Sigma_{r,M}(q)$.
Differentiating $\Phi_{r,M}(\Sigma_{r,M}(q),q)=0$ with respect to $q$ yields
$$
\Sigma_{r,M}'(q)
=
-
\frac{\partial_q \Phi_{r,M}(\Sigma_{r,M}(q),q)}
{\partial_\lambda \Phi_{r,M}(\Sigma_{r,M}(q),q)}.
$$
Now
$$
\partial_q \Phi_{r,M}(\lambda,q)
=
-\frac{c}{1-(c+1)\lambda}
+
\frac{c+1}{1-c\lambda}
=
\frac{1-(2c+1)\lambda}{(1-(c+1)\lambda)(1-c\lambda)}.
$$
Since $c\le r-1$, the denominator is positive on $(1/(r+1),1/r)$.
Hence the sign of $\partial_q\Phi_{r,M}$ is exactly the sign of $1-(2c+1)\lambda$.

Suppose first that $c\ge s=\lceil r/2\rceil$.
Then $
2c+1\ge 2\left\lceil \frac{r}{2}\right\rceil+1\ge r+1$,
so
$\frac{1}{2c+1}\le \frac{1}{r+1}<\Sigma_{r,M}(q).$
Therefore
$1-(2c+1)\Sigma_{r,M}(q)<0,$
hence $\partial_q\Phi_{r,M}<0$ and
$\Sigma_{r,M}'(q)>0.$
So $\Sigma_{r,M}$ is strictly increasing on every interval $I_c$ with $c\ge s$.

Suppose next that $c\le s-1$.
Then
$2c+1\le 2s-1\le r,$
and thus
$\Sigma_{r,M}(q)<\frac{1}{r}\le \frac{1}{2c+1}.$
Therefore
$1-(2c+1)\Sigma_{r,M}(q)>0,$
hence $\partial_q\Phi_{r,M}>0$ and
$\Sigma_{r,M}'(q)<0.$
So $\Sigma_{r,M}$ is strictly decreasing on every interval $I_c$ with $c\le s-1$.

It remains to prove continuity across the breakpoints
$q=\frac{M}{c+1}, (1\le c\le r-1).$
At such a point, the formulas from the two neighboring blocks coincide:
$$
\begin{aligned}
&\frac{1}{1-(r+1)\lambda}
+\frac{1}{1-r\lambda}
+\frac{M-c\frac{M}{c+1}}{1-(c+1)\lambda}
+\frac{\frac{M}{c+1}(c+1)-M}{1-c\lambda}\\
&\qquad=
\frac{1}{1-(r+1)\lambda}
+\frac{1}{1-r\lambda}
+\frac{\frac{M}{c+1}}{1-(c+1)\lambda},
\end{aligned}
$$
while
$$
\begin{aligned}
&\frac{1}{1-(r+1)\lambda}
+\frac{1}{1-r\lambda}
+\frac{M-(c+1)\frac{M}{c+1}}{1-(c+2)\lambda}
+\frac{\frac{M}{c+1}(c+2)-M}{1-(c+1)\lambda}\\
&\qquad=
\frac{1}{1-(r+1)\lambda}
+\frac{1}{1-r\lambda}
+\frac{\frac{M}{c+1}}{1-(c+1)\lambda}.
\end{aligned}
$$
Hence the two neighboring branches meet. Together with the endpoint case $q=M/r$ treated above, this shows that $\Sigma_{r,M}$ is continuous on
$[M/r,M]$.
Combining continuity with the monotonicity on each block shows that $\Sigma_{r,M}$ is
strictly increasing on $[M/r,M/s]$ and strictly decreasing on $[M/s,M]$.

Finally, for every feasible integer $q$, writing $M=qc+t$ with $0\le t<q$, we have
$\Sigma_{r,M}(q)=\lambda_2\bigl(\AS(r,q+2,c,t)\bigr),$
because \eqref{eq:Phi-rMq} is exactly the root equation \eqref{eq:balanced-root} for the generalized almost seesaw tree $\AS(r,q+2,c,t)$.
Therefore the strict unimodality of $\Sigma_{r,M}$ implies that the maximum over all feasible
integers is attained at one of the two integers nearest to $M/s$, namely $q_-$ or $q_+$.
This proves \eqref{eq:q-two-candidates}.
\end{proof}

\section{Boundary fluxes and two-center reduction}\label{sec:reduction}

\subsection{The inverse boundary quadratic form}
In this section we recast the problem in terms of boundary fluxes and the inverse Dirichlet-to-Neumann form. This is the discrete analogue of working with the inverse Steklov operator on the boundary, and it is particularly well adapted to diameter-constrained comparisons. The cut decomposition below may also be viewed as a boundary counterpart of the classical distance-matrix identities for trees; compare \cite{GrahamPollak1971,GrahamLovasz1978,Merris1990}.

Let $T=(V,E)$ be a tree and let $\delta\Omega\subset V$ be its leaf set.
Given a boundary flux vector $z\in \mathbb{R}^{\delta\Omega}$ with $\sum_{\omega\in \delta\Omega} z_\omega =0$, there exists a unique function $u_z\colon V\to \mathbb{R}$, unique up to an additive constant,
such that $u_z$ is harmonic on the interior vertices and its outward normal derivative on
$\delta\Omega$ is exactly $z$.
Define
$$
Q_T(z):=\sum_{xy\in E}\bigl(u_z(x)-u_z(y)\bigr)^2.
$$

\begin{proposition}\label{prop:inverse-rayleigh}
For every tree $T$ with leaf boundary $\delta\Omega$,
$$
\frac1{\lambda_2(T,\delta\Omega)}
=
\max_{\substack{z\in \mathbb{R}^{\delta\Omega}\setminus\{0\}\\ \sum_{\omega\in \delta\Omega} z_\omega=0}}
\frac{Q_T(z)}{\sum_{\omega\in \delta\Omega} z_\omega^2}.
$$
In other words, the nonzero Steklov eigenvalues are the reciprocals of the eigenvalues of the
quadratic form $Q_T$ on the codimension-one space
$\{z\in\mathbb{R}^{\delta\Omega}:\sum_{\omega\in\delta\Omega} z_\omega=0\}$.
\end{proposition}

\begin{proof}
Let $\Lambda_T$ be the Dirichlet-to-Neumann operator.
Restricted to the subspace of mean-zero boundary data, $\Lambda_T$ is positive definite and
invertible.
Its inverse sends a boundary flux vector $z$ to the corresponding boundary potential
$u_z|_{\delta\Omega}$, modulo constants.
By discrete Green's identity,
$$
Q_T(z)=\sum_{xy\in E}\bigl(u_z(x)-u_z(y)\bigr)^2
=\sum_{\omega\in \delta\Omega} u_z(\omega) z_\omega.
$$
Hence $Q_T$ is exactly the quadratic form of $\Lambda_T^{-1}$ on the mean-zero subspace.
Therefore the eigenvalues of $Q_T$ are $1/\lambda_k(T,\delta\Omega)$ for $k\ge 2$, and the largest one is
$1/\lambda_2(T,\delta\Omega)$.
This is precisely the stated variational characterization.
\end{proof}

We next show the cut-sum description of $Q_T$.
Fix a root $o$ of $T$ and orient every edge away from $o$.
For an oriented edge $e$, let $U_e$ denote the set of boundary leaves in the descendant
component of $e$, and let $s_e(z)=\sum_{\omega\in U_e} z_\omega.$

\begin{proposition}\label{prop:cutsum}
For every mean-zero boundary flux vector $z$, $
Q_T(z)=\sum_{e\in E} s_e(z)^2.$
In particular, if $D_{\delta\Omega}$ denotes the distance matrix of the leaves, then
$Q_T(z)=-\frac12 z^T D_{\delta\Omega} z.$
\end{proposition}

\begin{proof}
Let $e=xy$ be an oriented edge, with $x$ the parent and $y$ the child, and let $W_e$ be the
descendant component of $T-e$ containing $y$.
Summing $\Delta u_z(v)$ over all vertices $v\in W_e$, every edge internal to $W_e$ cancels, and the
only surviving contribution comes from the cut edge $e$:
$$
\sum_{v\in W_e}\Delta u_z(v)=u_z(y)-u_z(x).
$$
Now every non-leaf vertex of $W_e$ is interior to $T$ and therefore harmonic, while every leaf
$\omega\in U_e$ satisfies
$\Delta u_z(\omega)=\partial_n u_z(\omega)=z_\omega$
because the boundary is the leaf set.
Hence
$$
u_z(y)-u_z(x)=\sum_{\omega\in U_e} z_\omega=s_e(z).
$$
Therefore
$$
Q_T(z)=\sum_{e=xy\in E}\bigl(u_z(x)-u_z(y)\bigr)^2=\sum_{e\in E}s_e(z)^2,
$$
which proves the first identity.

For the distance-matrix formula, write
$$
d(\alpha,\beta)=\sum_{e\in E}\mathbf 1_{\{e\text{ separates }\alpha\text{ and }\beta\}}.
$$
Then
$$
z^T D_{\delta\Omega} z
=
\sum_{e\in E}
\sum_{\substack{\alpha\in U_e\\ \beta\notin U_e}} z_\alpha z_\beta
+
\sum_{e\in E}
\sum_{\substack{\alpha\notin U_e\\ \beta\in U_e}} z_\alpha z_\beta.
$$
For a fixed edge $e$, the two inner sums are equal, so its contribution is
$$
2\Bigl(\sum_{\alpha\in U_e} z_\alpha\Bigr)
 \Bigl(\sum_{\beta\notin U_e} z_\beta\Bigr).
$$
Because $\sum_{\omega\in\delta\Omega}z_\omega=0$, this equals
$$
-2\left(\sum_{\alpha\in U_e} z_\alpha\right)^2=-2\,s_e(z)^2.
$$
Summing over all edges and using the first identity gives
$$
z^T D_{\delta\Omega} z=-2\sum_{e\in E}s_e(z)^2=-2Q_T(z).
$$
\end{proof}

\subsection{Two-center profiles}

\begin{definition}
A \emph{double spider} $\DS(a_1,\dots,a_p\,;\,b_1,\dots,b_q)$ is a tree obtained by joining two distinguished vertices $u$ and $v$ by one edge $uv$, attaching $p$
internally disjoint paths of lengths $a_1,\dots,a_p$ to $u$, and attaching $q$ internally disjoint
paths of lengths $b_1,\dots,b_q$ to $v$. We also refer to this configuration as a
\emph{two-center profile}. We assume throughout that
$$
a_1\ge a_2\ge \cdots \ge a_p\ge 1,
\qquad
b_1\ge b_2\ge \cdots \ge b_q\ge 1.
$$
\end{definition}

If $a_1=b_1=r$, then $\DS(a_1,\dots,a_p\,;\,b_1,\dots,b_q)$ has odd diameter $2r+1$.
Geometrically, it is a two-center configuration whose boundary leaves split naturally into a $u$-side and a $v$-side.

\begin{lemma}\label{lem:DS-inverse}
Let
$D=\DS(a_1,\dots,a_p\,;\,b_1,\dots,b_q)$
and let
$z=(x_1,\dots,x_p,y_1,\dots,y_q)\in\mathbb{R}^{p+q}$
be a boundary vector with
$$
\sum_{i=1}^p x_i + \sum_{j=1}^q y_j = 0.
$$
Let $s=\sum_{i=1}^p x_i = -\sum_{j=1}^q y_j.$
Then
\begin{equation}\label{eq:DS-inverse}
Q_D(z)= s^2 + \sum_{i=1}^p a_i x_i^2 + \sum_{j=1}^q b_j y_j^2.
\end{equation}
Hence
$$
\frac{1}{\lambda_2(D)}
=
\max_{\substack{z\ne 0\\ \sum x_i+\sum y_j=0}}
\frac{s^2 + \sum_i a_i x_i^2 + \sum_j b_j y_j^2}{\sum_i x_i^2 + \sum_j y_j^2}.
$$
\end{lemma}

\begin{proof}
By Proposition~\ref{prop:cutsum}, each edge on the $i$-th branch emanating from $u$ has cut sum $x_i$,
so the whole branch contributes $a_i x_i^2$.
Similarly, each edge on the $j$-th branch emanating from $v$ has cut sum $y_j$, so that branch contributes
$b_j y_j^2$.
The central edge $uv$ has cut sum
$\sum_{i=1}^p x_i=s$,
hence contributes $s^2$.
Summing all edge contributions gives \eqref{eq:DS-inverse}.
\end{proof}

\begin{lemma}\label{lem:DS-root}
Let
$D=\DS(a_1,\dots,a_p\,;\,b_1,\dots,b_q)$
with
$a_1=b_1=r\ge a_i,b_j.$
Let
$$
A_\rho=\sum_{i=1}^p \frac{1}{\rho-a_i},
\qquad
B_\rho=\sum_{j=1}^q \frac{1}{\rho-b_j},
\qquad \rho>r.
$$
Then $\rho_D=1/\lambda_2(D)$ is the unique solution of
\begin{equation}\label{eq:DS-root}
\frac{1}{A_\rho}+\frac{1}{B_\rho}=1,
\qquad \rho>r.
\end{equation}
Moreover, every maximizing vector in Lemma~\ref{lem:DS-inverse} can be chosen in the form
$x_i>0\ \ (1\le i\le p)$, $y_j<0\ \ (1\le j\le q),$
or with all signs reversed.
\end{lemma}

\begin{proof}
Let
$\rho=\rho_D=\frac{1}{\lambda_2(D)}$
and let $z^\ast=(x,y)$ be a maximizer of the inverse Rayleigh quotient from
Lemma~\ref{lem:DS-inverse}.

We first show that $\rho>r$.
Indeed, because $a_i,b_j\le r$, any admissible vector with $s=0$ in
\eqref{eq:DS-inverse} satisfies
$Q_D(z)\le r\|z\|_2^2.$
On the other hand, the admissible vector defined by
$x_1=1, y_1=-1, x_i=0\ (i\ge2), y_j=0\ (j\ge2)$
has $s=1$, and Lemma~\ref{lem:DS-inverse} gives
$\frac{Q_D(z)}{\|z\|_2^2} = \frac{1+a_1+b_1}{2} = r+\frac12.$
Hence
$\rho\ge r+\frac12>r.$
In particular, the quantities
$$
A_\rho=\sum_{i=1}^p\frac1{\rho-a_i},
\qquad
B_\rho=\sum_{j=1}^q\frac1{\rho-b_j}
$$
are well defined and strictly positive.

Since every admissible vector with $s=0$ gives Rayleigh quotient at most $r<\rho$, a maximizing
vector must satisfy $s\ne 0$.
After changing the overall sign if necessary, we may assume $s>0$.

Consider the smooth extension
$$
\widetilde Q_D(x,y):=\left(\sum_{i=1}^p x_i\right)^2+
\sum_{i=1}^p a_i x_i^2+
\sum_{j=1}^q b_j y_j^2
$$
on $\mathbb{R}^{p+q}$.
On the constraint hyperplane
$$
H=\left\{(x,y)\in\mathbb{R}^{p+q}:\sum_{i=1}^p x_i+\sum_{j=1}^q y_j=0\right\},
$$
we have $\widetilde Q_D=Q_D$ by Lemma~\ref{lem:DS-inverse}.
After normalizing $z^\ast$ so that $\|z^\ast\|_2=1$, the vector $z^\ast$ maximizes
$\widetilde Q_D$ on the smooth compact manifold $H\cap S^{p+q-1}$.
Hence there exist Lagrange multipliers $\alpha$ and $\beta$ such that differentiating
$$
\widetilde Q_D(z)-\alpha\|z\|_2^2-\beta\Bigl(\sum_i x_i+\sum_j y_j\Bigr)
$$
at $z^\ast$ gives
$2s+2a_i x_i-2\alpha x_i-\beta=0, 2b_j y_j-2\alpha y_j-\beta=0.$
Taking the inner product with $z^\ast$ and using $\|z^\ast\|_2=1$,
$\sum_i x_i+\sum_j y_j=0$, and $\widetilde Q_D(z^\ast)=Q_D(z^\ast)=\rho$, we obtain
$2\rho=2\alpha.$
Thus $\alpha=\rho$, and hence
\begin{equation}\label{eq:lmult-DS}
(\rho-a_i)x_i=s-\frac{\beta}{2},
\qquad
(\rho-b_j)y_j=-\frac{\beta}{2}.
\end{equation}
Summing over $i$ and $j$ yields
$s=\Bigl(s-\frac{\beta}{2}\Bigr)A_\rho, -s=-\frac{\beta}{2}B_\rho.$
From the second identity,
$-\frac{\beta}{2}=-\frac{s}{B_\rho}<0.$
Substituting into the first gives
$s=\Bigl(s-\frac{s}{B_\rho}\Bigr)A_\rho =sA_\rho\Bigl(1-\frac1{B_\rho}\Bigr).$
Since $s\ne0$, this is equivalent to
$\frac1{A_\rho}+\frac1{B_\rho}=1.$
Thus any maximizer gives a solution of \eqref{eq:DS-root}.

Now consider the function
$\Phi(\rho):=\frac1{A_\rho}+\frac1{B_\rho}, \rho>r.$
It is strictly increasing, because $A_\rho$ and $B_\rho$ are strictly decreasing.
Moreover,
$$
\lim_{\rho\to r^+}\Phi(\rho)=0,
\qquad
\lim_{\rho\to\infty}\Phi(\rho)=+\infty.
$$
Hence \eqref{eq:DS-root} has exactly one solution in $(r,\infty)$.
Let $\rho_0$ denote this unique solution.
Since $\rho_D$ yields a solution, we have $\rho_D=\rho_0$.

Conversely, define
$x_i=\frac{1/A_{\rho_0}}{\rho_0-a_i}, y_j=-\frac{1/B_{\rho_0}}{\rho_0-b_j}.$
Then $\sum_i x_i=1=-\sum_j y_j$, so $z=(x,y)$ is admissible.
Moreover, \eqref{eq:lmult-DS} holds with $s=1$ and $-\beta/2=-1/B_{\rho_0}$.
Using \eqref{eq:DS-root}, we compute
\begin{align*}
\rho_0\|z\|_2^2
&= \sum_{i=1}^p \rho_0 x_i^2 + \sum_{j=1}^q \rho_0 y_j^2 \\
&= \sum_{i=1}^p a_i x_i^2 + \sum_{j=1}^q b_j y_j^2
   + \frac1{A_{\rho_0}}\sum_{i=1}^p x_i - \frac1{B_{\rho_0}}\sum_{j=1}^q y_j \\
&= \sum_{i=1}^p a_i x_i^2 + \sum_{j=1}^q b_j y_j^2 + \frac1{A_{\rho_0}}+\frac1{B_{\rho_0}} \\
&= 1 + \sum_{i=1}^p a_i x_i^2 + \sum_{j=1}^q b_j y_j^2 \\
&= Q_D(z).
\end{align*}
Thus the Rayleigh quotient of $z$ equals $\rho_0=\rho_D$, so $z$ is a maximizing vector.

Finally, return to the maximizing vector $z^\ast$.
Since $\rho=\rho_D>r\ge a_i,b_j$, all denominators in \eqref{eq:lmult-DS} are positive.
Moreover,
$-\frac{\beta}{2}=-\frac{s}{B_\rho}<0, s-\frac{\beta}{2}=\frac{s}{A_\rho}>0.$
Hence every $x_i$ is positive and every $y_j$ is negative.
Reversing the overall sign yields the alternative sign pattern.
This proves the sign claim.
\end{proof}

We now compare an arbitrary odd-diameter tree with a suitable two-center profile.

\begin{theorem}\label{thm:DS-domination}
Let $T$ be a tree with
$\diam(T)=2r+1, |V(T)|=n.$
Then there exists a double spider
$D=\DS(a_1,\dots,a_p\,;\,b_1,\dots,b_q)$
of the same order and diameter such that
\begin{enumerate}[label=\textnormal{(\roman*)}]
    \item $a_1=b_1=r$;
    \item every $a_i$ and $b_j$ is at most $r$;
    \item $\lambda_2(T)\le \lambda_2(D)$.
\end{enumerate}
Moreover, if equality holds, then $T$ itself is a double spider.
\end{theorem}

\begin{proof}
Let $uv$ be the central edge of $T$. Removing $uv$ splits $T$ into two rooted components
$R_u$ and $R_v$, rooted at $u$ and $v$ respectively. Because $\diam(T)=2r+1$, each of these
components has depth exactly $r$.

\smallskip
\noindent\emph{Step 1: Construction of a two-center comparison profile.}
Let
$$
L_u=\{\omega_1,\dots,\omega_p\},
\qquad
L_v=\{\eta_1,\dots,\eta_q\}
$$
be the leaf sets of $R_u$ and $R_v$. Choose deepest leaves $\omega_1\in L_u$ and
$\eta_1\in L_v$ such that $d(u,\omega_1)=r$ and $d(v,\eta_1)=r$.

For every edge $e$ of $R_u$, choose a terminal leaf $\pi_u(e)\in L_u$ lying below $e$, with the
additional requirement that every edge on the geodesic from $u$ to $\omega_1$ is assigned to
$\omega_1$. Define
$$
a_i:=|\pi_u^{-1}(\omega_i)|,
\qquad i=1,\dots,p.
$$
Then $a_1=r$ and $1\le a_i\le r$ for every $i$. In the same way, assigning every edge of $R_v$
to a terminal leaf, with the full geodesic from $v$ to $\eta_1$ assigned to $\eta_1$, we obtain
integers
$$
b_j:=|\pi_v^{-1}(\eta_j)|,
\qquad j=1,\dots,q,
$$
with $b_1=r$ and $1\le b_j\le r$. Since $R_u$ and $R_v$ contain all edges of $T$ except the
central edge $uv$, we have
$$
\sum_{i=1}^p a_i+\sum_{j=1}^q b_j=|E(T)|-1=n-2.
$$
Therefore the double spider
$$
D:=\DS(a_1,\dots,a_p\,;\,b_1,\dots,b_q)
$$
has the same order as $T$, and its diameter is
$$
a_1+1+b_1=r+1+r=2r+1.
$$
Thus $D$ is a two-center profile with the same geometric constraints as $T$.

\smallskip
\noindent\emph{Step 2: Comparison of the inverse energies.}
Let $z=(x_1,\dots,x_p,y_1,\dots,y_q)$ be a maximizing boundary flux for $D$ given by
Lemma~\ref{lem:DS-root}, chosen so that $x_i>0$ and $y_j<0$. Write
$$
s=\sum_{i=1}^p x_i=-\sum_{j=1}^q y_j>0.
$$
Transfer this flux to the leaves of $T$ by assigning the value $x_i$ to $\omega_i$ and the value
$y_j$ to $\eta_j$. Denote the resulting boundary flux by $\widetilde z$. Then $\widetilde z$ is admissible
for $T$ and $\|\widetilde z\|_2^2=\|z\|_2^2$.

For an edge $e$ in $R_u$, all leaves below $e$ belong to $L_u$ and carry positive flux. Hence
$$
s_e(\widetilde z)=\sum_{\omega_i\in U_e}x_i\ge \widetilde z_{\pi_u(e)}>0,
\qquad
s_e(\widetilde z)^2\ge \widetilde z_{\pi_u(e)}^2.
$$
Summing over all edges of $R_u$ gives
$$
\sum_{e\in E(R_u)} s_e(\widetilde z)^2\ge \sum_{i=1}^p a_i x_i^2.
$$
Likewise, for an edge $e$ in $R_v$, all leaves below $e$ belong to $L_v$ and carry negative flux, so
$$
|s_e(\widetilde z)|=\sum_{\eta_j\in U_e}|y_j|\ge |\widetilde z_{\pi_v(e)}|,
\qquad
s_e(\widetilde z)^2\ge \widetilde z_{\pi_v(e)}^2,
$$
and therefore
$$
\sum_{e\in E(R_v)} s_e(\widetilde z)^2\ge \sum_{j=1}^q b_j y_j^2.
$$
Finally, the central edge $uv$ separates the two sides of the tree and contributes exactly $s^2$.
Using Proposition~\ref{prop:cutsum}, we obtain
$$
Q_T(\widetilde z)
\ge
s^2+\sum_{i=1}^p a_i x_i^2+\sum_{j=1}^q b_j y_j^2
=
Q_D(z).
$$
Hence
$$
\frac1{\lambda_2(T)}
\ge
\frac{Q_T(\widetilde z)}{\|\widetilde z\|_2^2}
\ge
\frac{Q_D(z)}{\|z\|_2^2}
=
\frac1{\lambda_2(D)}.
$$
Equivalently, $\lambda_2(T)\le \lambda_2(D)$.

\smallskip
\noindent\emph{Step 3: Rigidity in the equality case.}
Assume now that equality holds. Then every inequality above is an equality. Since each $x_i>0$
and each $y_j<0$, equality for a fixed edge $e\in E(R_u)$ forces the descendant set below $e$ to
contain only one leaf of $L_u$; otherwise the strict positivity of the $x_i$ would give
$$
s_e(\widetilde z)^2=
\left(\sum_{\omega_i\in U_e}x_i\right)^2>
\widetilde z_{\pi_u(e)}^2.
$$
Thus every edge of $R_u$ lies below exactly one leaf, which means that all branching in $R_u$
occurs at the root $u$. The same argument applies to $R_v$. Hence $T$ itself is a double spider.
\end{proof}

The next lemma is the key monotonicity statement.
It says that if both sides of an odd-diameter double spider carry more than one branch, then one can move any non-principal branch from the smaller side to the larger side and strictly increase $\lambda_2$.

\begin{lemma}\label{lem:arm-transfer}
Let
$D=\DS(a_1,\dots,a_p\,;\,b_1,\dots,b_q)$
with
$a_1=b_1=r, p\ge 1, q\ge 2.$
Let
$\rho_D=\frac{1}{\lambda_2(D)}$
and write
$$
A=\sum_{i=1}^p \frac{1}{\rho_D-a_i},
\qquad
B=\sum_{j=1}^q \frac{1}{\rho_D-b_j}.
$$
Assume that
$A\ge B.$
Fix an index $k\in\{2,\dots,q\}$ and move the branch $b_k$ from the $v$-side to the $u$-side:
$D'=\DS(a_1,\dots,a_p,b_k\,;\,b_1,\dots,\widehat{b_k},\dots,b_q).$
Then $D'$ has the same order and diameter as $D$, and
$\lambda_2(D')>\lambda_2(D).$
\end{lemma}

\begin{proof}
Because $b_k\le b_1=r$, moving this branch across the central edge does not change the
diameter: the two principal branches of lengths $a_1=r$ and $b_1=r$ remain in place, hence
$\diam(D')=r+1+r=2r+1=\diam(D).$
The order is unchanged because no vertex is created or deleted.

Let $u=\frac{1}{\rho_D-b_k},
Q=\sum_{j\ne k} \frac{1}{\rho_D-b_j}.$
Then
$B=Q+u.$
By Lemma~\ref{lem:DS-root}, the inverse eigenvalue of $D$ is characterized by
$\frac1{A}+\frac1{Q+u}=1.$
Now consider the root function for $D'$:
$$
\Psi_{D'}(t)=
\frac{1}{\displaystyle \sum_{i=1}^p \frac1{t-a_i}+\frac1{t-b_k}}
+
\frac{1}{\displaystyle \sum_{j\ne k} \frac1{t-b_j}}
-1.
$$
By Lemma~\ref{lem:DS-root}, $\Psi_{D'}$ is strictly increasing on $(r,\infty)$, and its unique root is
$\rho_{D'}=\frac{1}{\lambda_2(D')}.$
Evaluating at $t=\rho_D$ gives
$\Psi_{D'}(\rho_D)=\frac1{A+u}+\frac1Q-1.$
Subtracting the old equation yields
$$
\Psi_{D'}(\rho_D)-\left(\frac1A+\frac1{Q+u}-1\right)
=
\left(\frac1{A+u}-\frac1A\right)
+
\left(\frac1Q-\frac1{Q+u}\right).
$$
Thus
$$
\Psi_{D'}(\rho_D)
=u\left(\frac1{Q(Q+u)}-\frac1{A(A+u)}\right).
$$
Since $A\ge B=Q+u>Q$, we have
$\frac1{Q(Q+u)}>\frac1{A(A+u)},$
so $\Psi_{D'}(\rho_D)>0$.
Because $\Psi_{D'}$ is strictly increasing and vanishes at $\rho_{D'}$, it follows that
$\rho_{D'}<\rho_D.$
Equivalently,
$\lambda_2(D')>\lambda_2(D).$
\end{proof}

\subsection{Rigidity of extremizers}
We can now prove the desired structural theorem.

\begin{theorem}\label{thm:extremizer-spider}
Let $D=2r+1\ge 3$ and let $n\ge D+1$.
Suppose that $T$ maximizes $\lambda_2$ among all trees with
$|V(T)|=n, \diam(T)=D.$
Then $T$ is a spider.
\end{theorem}

\begin{proof}
By Theorem~\ref{thm:DS-domination}, there exists a double spider
$$
D_0=\DS(a_1,\dots,a_p\,;\,b_1,\dots,b_q)
$$
of the same order and diameter as $T$ such that $\lambda_2(T)\le \lambda_2(D_0)$. Since $T$ is
globally extremal, equality must hold. The rigidity statement in Theorem~\ref{thm:DS-domination}
therefore implies that $T$ itself is a double spider.

Write
$$
T=\DS(a_1,\dots,a_p\,;\,b_1,\dots,b_q)
$$
with $a_1=b_1=r$. If both $p\ge2$ and $q\ge2$, then at the inverse eigenvalue
$\rho_T=1/\lambda_2(T)$ one of the two sums
$$
\sum_{i=1}^p \frac1{\rho_T-a_i},
\qquad
\sum_{j=1}^q \frac1{\rho_T-b_j}
$$
is at least the other. After relabeling the two sides if necessary, Lemma~\ref{lem:arm-transfer}
transports a nonprincipal branch from the smaller side to the larger one and strictly increases
$\lambda_2$, contradicting extremality. Thus one side of $T$ consists of a single branch.

A double spider with one side reduced to a single branch has at most one branching point. Hence it
is a spider.
\end{proof}

\section{Global classification and the large-order regime}\label{sec:global-classification}

\subsection{Global two-profile reduction}
We now combine the spider optimization from the first part of the paper with the structural reduction
from Theorem~\ref{thm:extremizer-spider}.

\begin{theorem}\label{thm:global-odd}
Let $D=2r+1\ge 3$ and let $n\ge D+1$.
Let $M=n-2r-2.$ Then the following hold.
\begin{enumerate}[label=\textnormal{(\roman*)}]
    \item If $M=0$, then the unique maximizer of $\lambda_2$ among all trees with order $n$ and
    diameter $D$ is the path $P_{D+1}$.
    \item Assume $M\ge 1$, and define
    $$
    s=\left\lceil \frac{r}{2}\right\rceil,
    \qquad
    q_-=\max\!\left\{1,\left\lfloor \frac{M}{s}\right\rfloor\right\},
    \qquad
    q_+=\left\lceil \frac{M}{s}\right\rceil.
    $$
    Write
    $$
    M=q_-c_-+t_-,\qquad 0\le t_-<q_-,
    $$
    $$
    M=q_+c_++t_+,\qquad 0\le t_+<q_+.
    $$
    Then
    \begin{equation}\label{eq:global-two-candidates}
    \max_{\substack{|V(T)|=n\\ \diam(T)=D}} \lambda_2(T)
    =
    \max\Bigl\{\lambda_2\bigl(\AS(r,q_-+2,c_-,t_-)\bigr),\,
    \lambda_2\bigl(\AS(r,q_++2,c_+,t_+)\bigr)\Bigr\}.
    \end{equation}
    Every maximizer is isomorphic either to $\AS(r,q_-+2,c_-,t_-)$ or to $\AS(r,q_++2,c_+,t_+)$.
    In particular, if the two values on the right-hand side of \eqref{eq:global-two-candidates}
    are different, then the maximizer is unique.
\end{enumerate}
\end{theorem}

\begin{proof}
If $M=0$, then $n=D+1$, so the only tree with order $n$ and diameter $D$ is the path $P_{D+1}$.

Assume now that $M\ge 1$.
Let $T$ be a tree with $|V(T)|=n$ and $\diam(T)=D$ maximizing $\lambda_2$.
By Theorem~\ref{thm:extremizer-spider}, the tree $T$ must be a spider.
The exact spider optimization theorem, namely Theorem~\ref{thm:main}, therefore shows that
$T$ is a generalized almost seesaw tree $\AS(r,q+2,c,t)$ for some feasible integer $q$, where $M=qc+t$ and $0\le t<q$.
The unimodality theorem, Theorem~\ref{thm:q-unimodal}, then reduces the remaining optimization
over $q$ to the two adjacent values $q_-$ and $q_+$.
This proves that every maximizer is one of the two generalized almost seesaw trees in the statement, and that the
maximal value is exactly the right-hand side of \eqref{eq:global-two-candidates}.

Conversely, both generalized almost seesaw trees $\AS(r,q_-+2,c_-,t_-)$ and $\AS(r,q_++2,c_+,t_+)$ are admissible trees of order $n$
and diameter $D$, so the right-hand side of \eqref{eq:global-two-candidates} is always attained.
The uniqueness claim is immediate.
\end{proof}

\subsection{A single-threshold comparison}
We next reorganize the remaining comparison by the order parameter $M=n-2r-2$.

\begin{proposition}\label{prop:single-threshold}
Fix $r\ge 2$ and let $s=\lceil r/2\rceil$.
Let $1\le t\le s-1$, let $k\ge s$, and let $M=ks+t$.
Then
$$
q_-=\left\lfloor \frac{M}{s}\right\rfloor=k,
\qquad
q_+=\left\lceil \frac{M}{s}\right\rceil=k+1.
$$
Hence the two candidates from Theorem~\ref{thm:global-odd} are
$$
A_{k,t}=\AS(r,k+2,s,t)
=S\bigl(r+1,r,\underbrace{s+1,\dots,s+1}_{t},\underbrace{s,\dots,s}_{k-t}\bigr),
$$
and
$$
B_{k,t}=\AS(r,k+3,s-1,k+t-s+1)
=S\bigl(r+1,r,\underbrace{s,\dots,s}_{k+t-s+1},\underbrace{s-1,\dots,s-1}_{s-t}\bigr).
$$
Let $F_{A_{k,t}}$ and $F_{B_{k,t}}$ be the corresponding root functions from Lemma~\ref{lem:spider-equation}.
Then, on $I_r=\left(\frac1{r+1},\frac1r\right)$,
we have
\begin{equation}\label{eq:single-threshold-difference}
F_{B_{k,t}}(\lambda)-F_{A_{k,t}}(\lambda)
=
\frac{P_{s,t}(\lambda)}{(1-(s-1)\lambda)(1-s\lambda)(1-(s+1)\lambda)},
\end{equation}
where
$P_{s,t}(\lambda)=1-3s\lambda+(2s^2+s-2t-1)\lambda^2.$
Moreover, the following hold.
\begin{enumerate}[label=\textnormal{(\roman*)}]
    \item If $r=2s$ and $2t\le s-1$, then
    $$
    \lambda_2(A_{k,t})>\lambda_2(B_{k,t})
    \qquad\text{for every }k\ge s.
    $$
    \item If $r=2s-1$ and $2t\ge s-1$, then
    $$
    \lambda_2(B_{k,t})>\lambda_2(A_{k,t})
    \qquad\text{for every }k\ge s.
    $$
    \item In the remaining cases, namely
    $$
    r=2s,\ 2t\ge s,
    \qquad\text{or}\qquad
    r=2s-1,\ 2t\le s-2,
    $$
    the polynomial $P_{s,t}$ has a unique root $\zeta_{s,t}$ in $I_r$.
    Define
    $$
    \kappa_{r,t}
    =-(1-s\zeta_{s,t})\left(
    \frac{1}{1-(r+1)\zeta_{s,t}}
    +\frac{1}{1-r\zeta_{s,t}}
    +\frac{t-s+1}{1-s\zeta_{s,t}}
    +\frac{s-t}{1-(s-1)\zeta_{s,t}}
    \right).
    $$
    Then, for every $k\ge s$,
    $$
    \lambda_2(A_{k,t})>\lambda_2(B_{k,t})\iff k>\kappa_{r,t},
    $$
    $$
    \lambda_2(B_{k,t})>\lambda_2(A_{k,t})\iff k<\kappa_{r,t},
    $$
    and equality holds if and only if $k=\kappa_{r,t}\in\mathbb Z$.
    In particular, for each fixed residue class $t$, there is at most one value of $k$ for which the two candidates tie.
\end{enumerate}
\end{proposition}

\begin{proof}
The identities $q_-=k$ and $q_+=k+1$ are immediate from $M=ks+t$ with $1\le t\le s-1$.
The displayed forms of $A_{k,t}$ and $B_{k,t}$ follow directly from the definition of the generalized almost seesaw trees.

Write
$$
\alpha_{k,t}=\lambda_2(A_{k,t}),
\qquad
\beta_{k,t}=\lambda_2(B_{k,t}).
$$
By Lemma~\ref{lem:spider-equation}, both $F_{A_{k,t}}$ and $F_{B_{k,t}}$ are strictly increasing on $I_r$, and $\alpha_{k,t},\beta_{k,t}$ are their unique roots there.

A direct computation gives
\begin{align*}
F_{A_{k,t}}(\lambda)
&=
\frac{1}{1-(r+1)\lambda}
+\frac{1}{1-r\lambda}
+\frac{t}{1-(s+1)\lambda}
+\frac{k-t}{1-s\lambda},\\
F_{B_{k,t}}(\lambda)
&=
\frac{1}{1-(r+1)\lambda}
+\frac{1}{1-r\lambda}
+\frac{k+t-s+1}{1-s\lambda}
+\frac{s-t}{1-(s-1)\lambda}.
\end{align*}
Subtracting the two expressions yields \eqref{eq:single-threshold-difference}.

Since $s\le r$, all three denominator factors in \eqref{eq:single-threshold-difference} are positive on $I_r$.
Hence the sign of $F_{B_{k,t}}-F_{A_{k,t}}$ on $I_r$ is exactly the sign of $P_{s,t}$.

We first show that $P_{s,t}$ is strictly decreasing on $I_r$.
Indeed,
$$
P'_{s,t}(\lambda)=-3s+2(2s^2+s-2t-1)\lambda.
$$
If $r=2s$, then $\lambda<1/r=1/(2s)$, and therefore
$$
P'_{s,t}(\lambda)
<
-3s+\frac{2(2s^2+s-2t-1)}{2s}
=
-s+1-\frac{2t+1}{s}
<0.
$$
If $r=2s-1$, then $\lambda<1/r=1/(2s-1)$, and hence
$$
P'_{s,t}(\lambda)
<
-3s+\frac{2(2s^2+s-2t-1)}{2s-1}
=
\frac{-2s^2+5s-4t-2}{2s-1}
<0.
$$
Thus $P_{s,t}$ is strictly decreasing on $I_r$ in all cases.

We now evaluate $P_{s,t}$ at the endpoints of $I_r$.

\medskip
\noindent
{\it Case 1: $r=2s$.}
Then
$$
P_{s,t}\!\left(\frac1{2s+1}\right)=\frac{2(s-t)}{(2s+1)^2}>0,
$$
and
$$
P_{s,t}\!\left(\frac1{2s}\right)=\frac{s-2t-1}{4s^2}.
$$
If $2t\le s-1$, then both endpoint values are nonnegative, and strict decrease implies
$$
P_{s,t}(\lambda)>0\qquad \text{for all }\lambda\in I_r.
$$
Hence $F_{B_{k,t}}(\lambda)-F_{A_{k,t}}(\lambda)>0$ on $I_r$.
Evaluating at $\lambda=\beta_{k,t}$ gives
$$
0-F_{A_{k,t}}(\beta_{k,t})>0,
$$
so $F_{A_{k,t}}(\beta_{k,t})<0$.
Since $F_{A_{k,t}}$ is strictly increasing and vanishes at $\alpha_{k,t}$, we obtain $\beta_{k,t}<\alpha_{k,t}$, that is,
$$
\lambda_2(A_{k,t})>\lambda_2(B_{k,t}).
$$
This proves \textnormal{(i)}.

If instead $2t\ge s$, then
$$
P_{s,t}\!\left(\frac1{2s+1}\right)>0,
\qquad
P_{s,t}\!\left(\frac1{2s}\right)<0.
$$
Since $P_{s,t}$ is strictly decreasing, it has a unique root $\zeta_{s,t}$ in $I_r$.

\medskip
\noindent
{\it Case 2: $r=2s-1$.}
Then
$$
P_{s,t}\!\left(\frac1{2s}\right)=\frac{s-2t-1}{4s^2},
$$
while
$$
P_{s,t}\!\left(\frac1{2s-1}\right)= -\frac{2t}{(2s-1)^2}<0.
$$
If $2t\ge s-1$, then both endpoint values are nonpositive, and strict decrease implies
$$
P_{s,t}(\lambda)<0\qquad \text{for all }\lambda\in I_r.
$$
Hence $F_{B_{k,t}}(\lambda)-F_{A_{k,t}}(\lambda)<0$ on $I_r$.
Evaluating at $\lambda=\alpha_{k,t}$ gives
$$
F_{B_{k,t}}(\alpha_{k,t})-0<0,
$$
so $F_{B_{k,t}}(\alpha_{k,t})<0$.
Since $F_{B_{k,t}}$ is strictly increasing and vanishes at $\beta_{k,t}$, we get $\alpha_{k,t}<\beta_{k,t}$, that is,
$$
\lambda_2(B_{k,t})>\lambda_2(A_{k,t}).
$$
This proves \textnormal{(ii)}.

If instead $2t\le s-2$, then
$$
P_{s,t}\!\left(\frac1{2s}\right)>0,
\qquad
P_{s,t}\!\left(\frac1{2s-1}\right)<0,
$$
so again $P_{s,t}$ has a unique root $\zeta_{s,t}$ in $I_r$.

\medskip
We are thus left with the two threshold cases, in which $P_{s,t}$ has a unique root $\zeta_{s,t}\in I_r$.
Since $P_{s,t}$ is strictly decreasing, we have
$$
P_{s,t}(\lambda)>0\iff \lambda<\zeta_{s,t},
\qquad
P_{s,t}(\lambda)<0\iff \lambda>\zeta_{s,t}.
$$
Equivalently,
$$
F_{B_{k,t}}(\lambda)-F_{A_{k,t}}(\lambda)>0\iff \lambda<\zeta_{s,t}.
$$

Now
$$
\lambda_2(A_{k,t})>\lambda_2(B_{k,t})
\iff \alpha_{k,t}>\beta_{k,t}
\iff F_{A_{k,t}}(\beta_{k,t})<0
\iff F_{B_{k,t}}(\beta_{k,t})-F_{A_{k,t}}(\beta_{k,t})>0,
$$
and hence
$$
\lambda_2(A_{k,t})>\lambda_2(B_{k,t})\iff \beta_{k,t}<\zeta_{s,t}.
$$
Since $F_{B_{k,t}}$ is strictly increasing and vanishes at $\beta_{k,t}$, this is equivalent to
$$
F_{B_{k,t}}(\zeta_{s,t})>0.
$$
Likewise,
$$
\lambda_2(B_{k,t})>\lambda_2(A_{k,t})\iff F_{B_{k,t}}(\zeta_{s,t})<0,
$$
and equality holds if and only if $F_{B_{k,t}}(\zeta_{s,t})=0$.

Finally, we compute $F_{B_{k,t}}(\zeta_{s,t})$:
\begin{align*}
F_{B_{k,t}}(\zeta_{s,t})
&=
\frac{1}{1-(r+1)\zeta_{s,t}}
+\frac{1}{1-r\zeta_{s,t}}
+\frac{k+t-s+1}{1-s\zeta_{s,t}}
+\frac{s-t}{1-(s-1)\zeta_{s,t}}\\
&=
\frac{k}{1-s\zeta_{s,t}}
+\left(
\frac{1}{1-(r+1)\zeta_{s,t}}
+\frac{1}{1-r\zeta_{s,t}}
+\frac{t-s+1}{1-s\zeta_{s,t}}
+\frac{s-t}{1-(s-1)\zeta_{s,t}}
\right)\\
&=
\frac{k-\kappa_{r,t}}{1-s\zeta_{s,t}}.
\end{align*}
Because $1-s\zeta_{s,t}>0$, the sign of $F_{B_{k,t}}(\zeta_{s,t})$ is exactly the sign of $k-\kappa_{r,t}$.
Therefore
$$
\lambda_2(A_{k,t})>\lambda_2(B_{k,t})\iff k>\kappa_{r,t},
$$
$$
\lambda_2(B_{k,t})>\lambda_2(A_{k,t})\iff k<\kappa_{r,t},
$$
and equality holds if and only if $k=\kappa_{r,t}\in\mathbb Z$.
This proves \textnormal{(iii)}.
\end{proof}

\subsection{Order-by-order classification}
The case $r=1$ (that is, diameter $3$) is immediate: here $s=1$ and $q_-=q_+=M$, so the unique extremizer is
$$
\AS(1,M+2,1,0)=S(2,1,\underbrace{1,\dots,1}_{M}).
$$
Hence we may assume $r\ge 2$ in the order-by-order discussion below.

\begin{theorem}\label{thm:all-orders}
Fix $r\ge 2$, let $s=\lceil r/2\rceil$, and let $n\ge 2r+2$.
Write $M=n-2r-2$.
Then the extremal problem for trees with order $n$ and diameter $2r+1$ falls into the following cases.
\begin{enumerate}[label=\textnormal{(\roman*)}]
    \item If $M=0$, then the unique extremizer is the path $P_{2r+2}$.
    \item If $1\le M<s$, then $q_-=q_+=1$, and the unique extremizer is
    $$
    \AS(r,3,M,0)=S(r+1,r,M).
    $$
    \item If $s\mid M$, then $q_-=q_+=M/s$, and the unique extremizer is
    $$
    \AS\!\left(r,\frac{M}{s}+2,s,0\right).
    $$
    \item Suppose $M=ks+t$ with $1\le t\le s-1$ and $k\ge s$.
    Then the only candidates are $A_{k,t}=\AS(r,k+2,s,t)$ and $B_{k,t}=\AS(r,k+3,s-1,k+t-s+1)$ from Proposition~\ref{prop:single-threshold}.
    More precisely:
    \begin{enumerate}[label=\textnormal{(\alph*)}]
        \item if $r=2s$ and $2t\le s-1$, then the unique extremizer is $A_{k,t}$;
        \item if $r=2s-1$ and $2t\ge s-1$, then the unique extremizer is $B_{k,t}$;
        \item in the remaining threshold cases, the extremizer is $A_{k,t}$ when $k>\kappa_{r,t}$,
        the extremizer is $B_{k,t}$ when $k<\kappa_{r,t}$, and both are extremal when $k=\kappa_{r,t}\in\mathbb Z$.
    \end{enumerate}
    \item Suppose $M=ks+t$ with $1\le k\le s-1$ and $1\le t\le s-1$.
    Then the only candidates are the two generalized almost seesaw trees from Theorem~\ref{thm:global-odd}, namely
    $$
    \AS(r,q_-+2,c_-,t_-)
    \qquad\text{and}\qquad
    \AS(r,q_++2,c_+,t_+),
    $$
    and the winner is determined by comparing the two explicit root equations from Theorem~\ref{thm:main}.
    For fixed $r$, these are finitely many initial orders.
\end{enumerate}
\end{theorem}

\begin{proof}
Cases \textnormal{(i)}--\textnormal{(iii)} are immediate from Theorem~\ref{thm:global-odd}, since in each of them $q_-=q_+$.
Case \textnormal{(iv)} follows from Theorem~\ref{thm:global-odd} together with Proposition~\ref{prop:single-threshold}; indeed, for $M=ks+t$ with $k\ge s$ and $1\le t\le s-1$, the two candidates are exactly $A_{k,t}=\AS(r,k+2,s,t)$ and $B_{k,t}=\AS(r,k+3,s-1,k+t-s+1)$.
Case \textnormal{(v)} is again Theorem~\ref{thm:global-odd}; the finiteness statement follows from $1\le k\le s-1$ and $1\le t\le s-1$.
\end{proof}

\end{document}